\documentclass[12pt,a4paper]{article}
\usepackage{amsmath,amssymb,amsthm,bm}
\usepackage{graphicx}
\usepackage{txfonts}

\newtheorem{theorem}{Theorem}[section]
\newtheorem{proposition}[theorem]{Proposition}
\newtheorem{corollary}[theorem]{Corollary}
\newtheorem{lemma}[theorem]{Lemma}
\newtheorem{remark}[theorem]{Remark}

\newcommand{\1}{{\bm 1}}

\numberwithin{equation}{section}

\begin{document}

\begin{center}
\large\bf
Determining Finite Connected Graphs \\
Along the Quadratic Embedding Constants of Paths
\end{center}

\bigskip

\begin{center}
Edy Tri Baskoro \\
Combinatorial Mathematics Research Group \\
Faculty of Mathematics and Natural Sciences \\
Institut Teknologi Bandung \\
Jalan Ganesa 10 Bandung, Indonesia \\
ebaskoro@math.itb.ac.id
\\[6pt]
and 
\\[6pt]
Nobuaki Obata\\
Graduate School of Information Sciences \\
Tohoku University\\
Sendai 980-8579 Japan \\
obata@tohoku.ac.jp
\end{center}

\bigskip

\begin{quote}
\textbf{Abstract}\enspace
The QE constant of a finite connected graph $G$,
denoted by $\mathrm{QEC}(G)$, is by definition
the maximum of the quadratic function 
associated to the distance matrix 
on a certain sphere of codimension two.
We prove that the QE constants of paths $P_n$ form a strictly
increasing sequence converging to $-1/2$.
Then we formulate the problem of determining all the graphs $G$
satisfying $\mathrm{QEC}(P_n)\le\mathrm{QEC}(G)<
\mathrm{QEC}(P_{n+1})$.
The answer is given for $n=2$ and $n=3$ 
by exploiting forbidden subgraphs for $\mathrm{QEC}(G)<-1/2$
and the explicit QE constants of 
star products of the complete graphs.
\end{quote}

\begin{quote}
\end{quote}

\begin{quote}
\end{quote}

\section{Introduction}

Let $G=(V,E)$ be a finite connected graph with $|V|=n\ge2$
and $D=[d(i,j)]_{i,j\in V}$ the distance matrix of $G$.
The \textit{quadratic embedding constant} 
(\textit{QE constant} for short)
of $G$ is defined by
\begin{equation}\label{0eqn:def od QEC(G)}
\mathrm{QEC}(G)
=\max\{\langle f,Df \rangle\,;\, f\in C(V), \,
\langle f,f \rangle=1, \, \langle \1,f \rangle=0\},
\end{equation}
where $C(V)$ is the space of all $\mathbb{R}$-valued functions
on $V$, $\bm{1}\in C(V)$ the constant function
taking value 1, and $\langle\cdot,\cdot\rangle$ the 
canonical inner product.
The QE constant was first introduced for
the quantitative study of quadratic embedding of graphs
in Euclidean spaces \cite{Obata2017,Obata-Zakiyyah2018}.
In particular, a graph $G$ admits a quadratic embedding
(in this case we say that $G$ is \textit{of QE class})
if and only if $\mathrm{QEC}(G)\le0$.
Moreover, it is noteworthy that
$\mathrm{QEC}(G)\le0$ 
is equivalent to the positive definiteness
of the $Q$-matrix $Q=[q^{d(i,j)}]$ for all $0\le q\le 1$.
This property,
first proved by Haagerup \cite{Haagerup79}
for trees and later by Bo\.zejko \cite{Bozejko89}
for general star products,
has many applications in harmonic analysis and quantum probability,
see \cite{Bozejko88,Obata2007,Obata2011} and references cited therein.

It is also interesting to observe a close relation between 
the QE constants and the distance spectra.
In fact, for a finite connected graph we have
\[
\lambda_2(G)
\le \mathrm{QEC}(G)
<\lambda_1(G),
\]
where $\lambda_1(G)$ and $\lambda_2(G)$
are respectively the largest and 
the second largest eigenvalues of the distance matrix of $G$.
It is straightforward to see that
$\lambda_2(G)=\mathrm{QEC}(G)$ holds if the distance matrix
of $G$ has a constant row sum (in some literatures,
such a graph is called \textit{transmission regular}).
But the converse is not true as the paths $P_n$ with even $n$
provide counter-examples.
In this aspect characterization of graphs satisfying 
$\lambda_2(G)=\mathrm{QEC}(G)$ is an interesting problem,
as is suggested by the attempt of classifying graphs
in terms of the second largest eigenvalue $\lambda_2(G)$,
see \cite{Koolen-Shpectorov1994}.

In this paper, we initiate the project of characterizing  
finite connected grahs in terms of the QE constants.
Our idea is based on the fact that
the QE constants of paths form a strictly increasing sequence:
\begin{equation}\label{0eqn:basic scale}
\mathrm{QEC}(P_2)<
\mathrm{QEC}(P_3)<\cdots<
\mathrm{QEC}(P_n)<
\mathrm{QEC}(P_{n+1})<
\cdots\rightarrow -\frac12\,.
\end{equation}
Then a natural question arises to determine finite connected graphs
along the above QE constants.
More precisely, 
we are interested in the family of graphs $G$ satisfying
\begin{equation}\label{01eqn:determination of family}
\mathrm{QEC}(P_n)\le \mathrm{QEC}(G)<\mathrm{QEC}(P_{n+1}),
\qquad n\ge2.
\end{equation}
The main goal of this paper is to give the answer 
to the first two cases of $n=2,3$.

This paper is organized as follows: In Section 2 we give
a quick review on the QE constant, for more details see
\cite{MO-2018,Obata-Zakiyyah2018}.

In Section \ref{sec:QE Constants of Paths}
we derive a general criterion for the strict 
inequality
\[
\mathrm{QEC}(G)<\mathrm{QEC}(G\star K_{m+1}),
\]
where $G\star K_{m+1}$ is the star product,
namely, the graph obtained by joining a graph $G$ 
and the complete graph $K_{m+1}$ at a single vertex,
see Theorem \ref{02thm:strict inequality}.
We then prove that 
the QE constants of paths form a strictly increasing sequence
as in \eqref{0eqn:basic scale}, see Theorem \ref{03thm:QEC(Path)}.

In Section \ref{sec:Classification} we prove the main results.
Case of $n=2$ is simple,
in fact, condition \eqref{01eqn:determination of family}
characterizes the complete graphs, see Theorem \ref{05thm: G_2}.
For a general case the first useful result is that
any graph with $\mathrm{QEC}(G)<-1/2$ is
diamond-free, claw-free, $C_4$-free and $C_5$-free,
see Corollary \ref{04cor:forbidden subgraphs}.
Then, using the explicit values of $\mathrm{QEC}(K_m\star K_n)$
we obtain an explicit list for case of $n=3$,
that is, a series of graphs $K_n\star K_2$ with $n\ge2$ 
and one sporadic $K_3\star K_3$, see Theorem \ref{05thm: G_3}.
As a result, $\mathrm{QEC}(P_4)$ is the smallest accumulation
point of the QE constants.
We also provide examples of graphs $G$ satisfying
$\mathrm{QEC}(G)=\mathrm{QEC}(P_4)$.

\bigskip
{\bfseries Acknowledgements:}
NO thanks Institut Teknologi Bandung for their kind hospitality,
where this work was completed in March 2019.
The support by JSPS Open Partnership Joint Research Project
``Extremal graph theory, algebraic graph theory and mathematical
approach to network science'' (2017--18) is gratefully acknowledged.
He also thanks Professor J. Koolen for stimulating discussion.

\section{Quadratic Embedding Constants}

\subsection{Definition and Basic Properties}

A graph $G=(V,E)$ is a pair of 
a non-empty set $V$ of vertices
and a set $E$ of edges,
i.e., $E$ is a subset of $\{\{i,j\}\,;\, i,j\in V, i\neq j\}$.
A graph is called finite if $V$ is a finite set.
Throughout this paper by a graph we mean a finite graph.

If $\{i,j\}\in E$, we write $i\sim j$ for simplicity.
A finite sequence of vertices $i_0,i_1,\dots,i_m\in V$ is called
an \textit{$m$-step walk} if
$i_0\sim i_1\sim \dotsb\sim i_m$.
In that case we say that $i_0$ and $i_m$ are connected by a walk
of length $m$.
A graph is called \textit{connected} if 
any pair of vertices are connected by a walk.

Let $G=(V,E)$ be a connected graph.
For $i,j \in V$ with $i\neq j$ let $d(i,j)=d_G(i,j)$ denote
the length of a shortest walk connecting $i$ and $j$.
By definition we set $d(i,i)=0$.
Then $d(i,j)$ becomes a metric on $V$,
which we call the \textit{graph distance}.
The \textit{diameter} of $G$ is defined by
\[
\mathrm{diam}(G)=\max\{d(i,j)\,;\, i,j\in V\}.
\]
The \textit{distance matrix} of $G$ is defined by
\[
D=D_G=[d(i,j)]_{i,j\in V}\,.
\]
 
Let $G=(V,E)$ be a connected graph with $|V|\ge2$.
The \textit{quadratic embedding constant}
(\textit{QE constant} for short) of $G$ is defined by
\begin{equation}\label{2eqn:def od QEC(G)}
\mathrm{QEC}(G)
=\max\{\langle f,Df \rangle\,;\, f\in C(V), \,
\langle f,f\rangle=1, \, \langle \1,f \rangle=0\},
\end{equation}
where $C(V)$ is the space of all $\mathbb{R}$-valued
functions on $V$ and
$\langle\cdot,\cdot\rangle$ the canonical inner product on $C(V)$.
Furthermore, $\bm{1}$ is the constant function defined by
$\bm{1}(x)=1$ for all $x\in V$,
and $\langle \1,f \rangle=\sum_{x\in V} f(x)$.
Indeed, identifying $C(V)$ with $\mathbb{R}^n$, $n=|V|$,
we see that the domain 
\[
\{f\in C(V)\,;\, \langle f,f\rangle=1,\, 
\langle \1,f \rangle=0\}
\]
is a compact manifold (in fact, a sphere of $n-2$ dimension).
Hence the quadratic function $\langle f,Df\rangle$ attains
the maximum on the above domain.

\begin{proposition}\label{01prop:QE class}
Let $G=(V,E)$ be a connected graph with $|V|\ge2$,
and $D=[d(i,j)]$ the distance matrix. 
Then the following conditions are equivalent:
\begin{enumerate}
\item[\upshape (i)] $G$ is of QE class, that is,
there exist a Euclidean space $\mathcal{H}$
and a map $\varphi:V\rightarrow\mathcal{H}$ such that
\[
\|\varphi(i)-\varphi(j)\|^2=d(i,j),
\qquad i,j\in V.
\]
\item[\upshape (ii)] $D$ is conditionally negative definite,
that is,
\[
\langle f,Df\rangle\le0
\quad\text{for all $f\in C(V)$ with
$\langle\bm{1},f\rangle=0$.}
\]
\item[\upshape (iii)] $\mathrm{QEC}(G)\le0$.
\end{enumerate}
\end{proposition}

The map $\varphi:V\rightarrow\mathcal{H}$ in the above condition (i)
is called a \textit{quadratic embedding} of $G$.
The above result is essentially due to Schoenberg 
\cite{Schoenberg1935,Schoenberg1938}
and motivated us to introduce the QE constant.

The graphs of QE class include the complete graphs 
$K_n$ ($n\ge2$), paths $P_n$ ($n\ge2$), and cycles $C_n$ $(n\ge3)$.
In fact,
\begin{equation}\label{02eqn:QEC(K_n)}
\mathrm{QEC}(K_n)=-1, \qquad n\ge2,
\end{equation}
and
\begin{equation}\label{02eqn:QEC(C_n)}
\mathrm{QEC}(C_{2n+1})=-\frac{1}{4\cos^2\dfrac{\pi}{2n+1}}\,, 
\qquad
\mathrm{QEC}(C_{2n+2})=0,
\qquad n\ge1,
\end{equation}
while a closed expression for $\mathrm{QEC}(P_n)$ is not known.
It is also noted that
the QE constant of a tree is negative.
In fact, for any tree $G$ on $n$ vertices we have
\begin{equation}\label{02eqn:QEC(tree)}
\mathrm{QEC}(G)\le -\frac{2}{2n-3}\,,
\qquad n\ge3.
\end{equation}
However, \eqref{02eqn:QEC(tree)} is a rather rough estimate 
and its refinement is an interesting question, 
see \cite[Section 5]{MO-2018}.

\begin{proposition}\label{01prop:isometrically embedded subgraphs}
Let $G=(V,E)$ be a connected graph and 
$H=(W,F)$ a connected subgraph of $G$ with $|W|\ge2$.
If $H$ is isometrically embedded in $G$, i.e.,
\[
d_H(i,j)=d_G(i,j)
\quad\text{for all $i,j\in W$},
\]
then we have
\[
\mathrm{QEC}(H)\le \mathrm{QEC}(G).
\]
\end{proposition}

\begin{proof}
Take $f\in C(W)$ such that
\[
\mathrm{QEC}(H)=\langle f,D_Hf\rangle,
\quad
\langle f,f\rangle_W=1, 
\quad \langle \1,f \rangle_W=0,
\]
where $\langle\cdot,\cdot\rangle_W$ denotes 
the inner product on $C(W)$.
Define $\tilde{f}\in C(V)$ in such a way that
$\tilde{f}(x)=f(x)$ for $x\in W$ and $\Tilde{f}(x)=0$ otherwise.
Then $\tilde{f}$ satisfies
$\langle \tilde{f},\Tilde{f}\rangle_V=1$ 
and $\langle \1,\Tilde{f} \rangle_V=0$.
Since $H$ is isometrically embedded in $G$,
the distance matrix $D_H$ is a submatrix of $D_G$.
Hence,
\[
\mathrm{QEC}(H)=\langle f,D_Hf\rangle
=\langle \tilde{f},D_G\tilde{f}\rangle,
\]
where the last quantity 
is bounded by $\mathrm{QEC}(G)$ by definition.
\end{proof}

\begin{corollary}
Let $P_n$ be the path on $n$ vertices.
Then we have
\begin{equation}\label{01eqn:scale by paths}
\mathrm{QEC}(P_2)
\le \mathrm{QEC}(P_3)
\le\dotsb\le \mathrm{QEC}(P_n)
\le \mathrm{QEC}(P_{n+1})
\le\dotsb.
\end{equation}
\end{corollary}

\begin{corollary}\label{02cor:QEC and diameter}
Let $G=(V,E)$ be a connected graph with $|V|\ge2$.
\begin{enumerate}
\setlength{\itemsep}{0pt}
\item[\upshape (1)] If $\mathrm{diam}(G)\ge d$, then
$\mathrm{QEC}(P_{d+1})\le \mathrm{QEC}(G)$.
\item[\upshape (2)] If $\mathrm{QEC}(G)<\mathrm{QEC}(P_{d+1})$,
then $\mathrm{diam}(G)\le d-1$.
\end{enumerate}
\end{corollary}

The proofs are straightforward from
Proposition \ref{01prop:isometrically embedded subgraphs}.
In fact, as is shown 
in Subsection \ref{03subsec:QE Constants of Paths},
the inequalities in \eqref{01eqn:scale by paths} are strict.

Next we derive a useful criterion for isometric embedding. 

\begin{lemma}\label{03lem:isometric embedding}
Let $G=(V,E)$ be a connected graph and
$H=(W,F)$ a connected subgraph.
\begin{enumerate}
\setlength{\itemsep}{0pt}
\item[\upshape(1)] If $H$ is isometrically embedded,
then $H$ is an induced subgraph of $G$.
\item[\upshape(2)] If $H$ is an induced subgraph of $G$ and
$\mathrm{diam\,}(H)\le2$, then
$H$ is isometrically embedded in $G$.
\end{enumerate}
\end{lemma}

\begin{proof}
Let $d_G$ and $d_H$ be the graph distances of $G$ and $H$,
respectively.

(1) Let $i,j\in W$ and assume that they are adjacent in $G$.
Then $d_G(i,j)=1$ and by assumption we have $d_H(i,j)=1$,
which means that $i$ and $j$ are adjacent in $H$ too.
Therefore, $H$ is an induced subgraph of $G$.

(2) Let $i,j\in W$. Then $d_H(i,j)\le2$ by assumption.
If $d_H(i,j)=0$, then $i=j$ and hence $d_G(i,j)=0$.
Suppose that $d_H(i,j)=1$. 
Then $i$ and $j$ are adjacent in $H$, so are in $G$.
Hence $d_G(i,j)=1$.
Finally, suppose that $d_H(i,j)=2$.
Obviously, $i\neq j$ so that $1\le d_G(i,j)\le2$.
If $d_G(i,j)=1$, then $i$ and $j$ are adjacent in $G$
and so are in $H$ since $H$ is an induced subgraph.
Then we obtain $d_H(i,j)=1$, which is contradiction.
Therefore, we have $d_G(i,j)=2$.
Consequently, $d_H(i,j)=d_G(i,j)$ for all $i,j\in W$,
which means that $H$ is isometrically embedded in $G$.
\end{proof}

\begin{proposition}\label{02prop:QEC estimated from below by diam 2}
Let $G$ be a connected graph,
and $H$ a connected and induced subgraph of $G$.
If $\mathrm{diam}(H)\le2$, we have
\[
\mathrm{QEC}(H)\le\mathrm{QEC}(G).
\]
\end{proposition}

\begin{proof}
It follows from Lemma \ref{03lem:isometric embedding} (2) that
$H$ is isometrically embedded in $G$.
Then, by Proposition 
\ref{01prop:isometrically embedded subgraphs}
we see that $\mathrm{QEC}(H)\le\mathrm{QEC}(G)$.
\end{proof}

\subsection{Calculating QE Constants}
\label{02subsec:Calculating QE Constants}

Let $G$ be a connected graph on $V=\{1,2,\dots,n\}$
and identify $C(V)$ with $\mathbb{R}^n$ in a natural manner.
Recall that $\mathrm{QEC}(G)$ is the conditional 
maximum of the quadratic function $\langle f,Df \rangle$,
$f=[f_i]=[f(i)]\in C(V)\cong\mathbb{R}^n$, subject to 
\begin{align}
&\langle f,f\rangle=\sum_{i=1}^n f_i^2=1,
\label{02eqn:cond 1} \\
&\langle \1,f \rangle=\sum_{i=1}^n f_i=0.
\label{02eqn:cond 2}
\end{align}
The method of Lagrange multipliers is applied to
calculating QE constants.
For later use we review it quickly,
for more details see \cite{Obata-Zakiyyah2018}.

First we set
\begin{equation}\label{02eqn:def of F}
F(f,\lambda,\mu)
=\langle f,Df\rangle 
-\lambda(\langle f,f\rangle-1) 
-\mu \langle\1,f\rangle,
\end{equation}
where $f=[f_i]\in \mathbb{R}^n$, $\lambda\in\mathbb{R}$ and
$\mu\in\mathbb{R}$.
Since conditions \eqref{02eqn:cond 1} and \eqref{02eqn:cond 2}
define a sphere of $n-2$ dimension, which is smooth and compact,
the conditional maximum of $\langle f, Df\rangle$ under question 
is attained at a stationary points of $F(f,\lambda,\mu)$.

Let $\mathcal{S}$ be the set of stationary points of 
$F(f,\lambda,\mu)$, that is,
\[
\mathcal{S}
=\left\{(f=[f_i],\lambda,\mu)
\in\mathbb{R}^n\times\mathbb{R}\times\mathbb{R}\,,\,
\frac{\partial F}{\partial f_i}
=\frac{\partial F}{\partial \lambda}
=\frac{\partial F}{\partial \mu}=0
\right\}.
\]
Taking the derivatives of \eqref{02eqn:def of F}, we obtain
\[
\frac{\partial F}{\partial f_i}
=2\langle e_i,Df\rangle -2\lambda\langle e_i,f\rangle 
 -\mu \langle\1,e_i\rangle
 =\langle e_i,2(D-\lambda)f-\mu \1\rangle,
\]
where $\{e_i\}$ is the canonical basis of $\mathbb{R}^n$.
Hence 
${\partial F}/{\partial f_i}=0$ for all $1\le i\le n$ 
if and only if $2(D-\lambda)f-\mu \1=0$,
that is,
\begin{equation}\label{02eqn:cond 3}
(D-\lambda)f=\frac{\mu}{2}\,\1.
\end{equation}
Thus, $\mathcal{S}$ is the set of 
$(f,\lambda,\mu)
\in\mathbb{R}^n\times\mathbb{R}\times\mathbb{R}$ satisfying
\eqref{02eqn:cond 1}, \eqref{02eqn:cond 2} and \eqref{02eqn:cond 3}.
On the other hand,
for $(f,\lambda,\mu)\in \mathcal{S}$ we have
\begin{equation}
\langle f,Df\rangle
=\left\langle f,\lambda f +\frac{\mu}{2}\,\1 \right\rangle
=\lambda \langle f,f\rangle +\frac{\mu}{2} \langle f,\1\rangle
=\lambda.
\end{equation}
Thus we come to the following useful result.

\begin{proposition}\label{02prop:QEC}
Let $G$ be a connected graph on $n\ge3$ vertices
and $\mathcal{S}$ the set of stationary points of
$F(f,\lambda,\mu)$ defined by \eqref{02eqn:def of F}.
Then we have
\[
\mathrm{QEC}(G)
=\max\{\lambda\in\mathbb{R}\,;\, 
\text{$(f,\lambda,\mu)\in \mathcal{S}$ for some
$f\in\mathbb{R}^n$ and $\mu\in\mathbb{R}$}\}.
\]
\end{proposition}

\section{QE Constants of Paths}
\label{sec:QE Constants of Paths}

\subsection{A Criterion for
$\mathrm{QEC}(G)<\mathrm{QEC}(G\star K_{m+1})$}

Let $G_1$ and $G_2$ be two graphs with disjoint vertex sets.
Choose $o_1$ and $o_2$ as distinguished vertices
of $G_1$ and $G_2$, respectively.
A \textit{star product} of $G_1$ and $G_2$ 
with respect to $o_1$ and $o_2$ is (informally) defined to
be the graph obtained by joining $G_1$ and $G_2$ at
the distinguished vertices $o_1$ and $o_2$.
If there is no danger of confusion,
the star product is denoted simply by $G_1\star G_2$.

In this subsection we consider the case where
$G_1$ is an arbitrary connected graph 
and $G_2$ a complete graph.
To be precise, 
for $n\ge2$ and $m\ge1$ let $G=(V,E)$ be a connected graph 
on $V=\{1,2,\dots,n\}$
and $K_{m+1}$ the complete graph on $\{n,n+1,\dots,n+m\}$.
We set
\[
\Tilde{V}=V\cup\{n,n+1,\dots, n+m\}=\{1,2,\dots, n+m\},
\]
and 
\[
\Tilde{E}=E\cup\{\{i,j\}\,;\, n\le i<j\le n+m\}.
\]
Then $\Tilde{G}=(\Tilde{V},\Tilde{E})$ becomes 
the star product of $G$ and $K_{m+1}$,
which we denote simply by $\Tilde{G}=G\star K_{m+1}$.
Since $G$ is isometrically embedded in $\Tilde{G}$, 
it follows from Proposition 
\ref{01prop:isometrically embedded subgraphs} that
\begin{equation}\label{03eqn:main inequality}
\mathrm{QEC}(G)\le\mathrm{QEC}(\Tilde{G})
=\mathrm{QEC}(G\star K_{m+1}).
\end{equation}
We are interested in when the inequality 
\eqref{03eqn:main inequality} becomes strict.
\begin{figure}[htb]
\begin{center}
\vspace*{10pt}
\includegraphics[width=140pt]{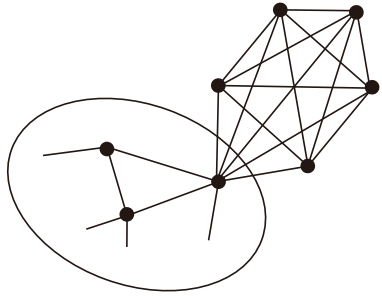}
\caption{$G \star K_{m+1}$ ($m=5$)}
\label{02fig:G star K_m+1}
\end{center}
\end{figure}

Let $D=D_G=[d(i,j)]$ and $\Tilde{D}=D_{\Tilde{G}}$ be 
the distance matrices of $G$ and $\Tilde{G}$, respectively.
Then we have
\begin{equation}\label{02eqn:D and tilde D}
\Tilde{D}
=\left[
\renewcommand{\arraystretch}{2}
\begin{array}{@{\,}ccc |ccc@{\,}}
 & D & & & S & \\ \hline
 & S^T & & & J-I & \\
\end{array}
\right],
\end{equation}
where $S=[s(i,j)]$ is the $n\times m$ matrix defined by
\begin{equation}\label{02eqn:def of S}
s(i,j)=d_G(i,n)+1,
\qquad
1\le i\le n, 
\quad
1\le j\le m,
\end{equation}
$J$ the matrix whose entries are all one
and $I$ the identity matrix.

\begin{theorem}\label{02thm:strict inequality}
Let $G$ be a connected graph on $V=\{1,2,\dots,n\}$
and $K_{m+1}$ the complete graph on $\{n,n+1,\dots,n+m\}$,
where $n\ge2$ and $m\ge1$.
Let $\Tilde{G}=(\Tilde{V},\Tilde{E})=G\star K_{m+1}$ 
be the star product defined as above.
If $\mathrm{QEC}(G)<0$ and 
there exists $f_0\in C(V)$ such that
\begin{equation}\label{02eqn:in main theorem}
\mathrm{QEC}(G)=\langle f_0, Df_0\rangle,
\quad
\langle f_0, f_0\rangle=1,
\quad
\langle \bm{1}, f_0\rangle=0
\end{equation}
and
\begin{equation}\label{02eqn:crucial condition}
f_0(n)\neq0,
\end{equation}
then we have
\begin{equation}\label{02eqn:strict inequality}
\mathrm{QEC}(G)<\mathrm{QEC}(G\star K_{m+1})<0.
\end{equation}
\end{theorem}

\begin{proof}[Proof of the left-half of
\eqref{02eqn:strict inequality}]
For simplicity we set $\lambda_0=\mathrm{QEC}(G)$.
Then taking $f_0\in C(V)\cong \mathbb{R}^n$ as in the 
above statement, we have
\begin{equation}\label{02eqn:in proof 3.1 (1)}
\lambda_0=\langle f_0, Df_0\rangle<0,
\end{equation}
and 
\begin{equation}\label{02eqn:in proof 3.1 (12)}
\langle f_0, f_0\rangle=1,
\qquad
\langle \bm{1}, f_0\rangle=0.
\end{equation}
(In fact, existence of $f_0$ satisfying \eqref{02eqn:in main theorem}
follows from the definition of QE constant.
The essential assumption is \eqref{02eqn:crucial condition}.)
On the other hand,
$\mathrm{QEC}(\Tilde{G})$ is
given by the conditional maximum of the quadratic function:
\[
\Phi
=\langle \Tilde{f}, \Tilde{D}\Tilde{f}\rangle,
\quad
\Tilde{f}\in C(\Tilde{V})\cong\mathbb{R}^{n+m},
\]
subject to
\begin{equation}\label{02eqn:in proof 3.1 (2)}
\langle \Tilde{f},\Tilde{f}\rangle=1,
\qquad
\langle\bm{1},\Tilde{f}\rangle=0.
\end{equation}
It is convenient to use new variables $(\xi,\eta)\in
\mathbb{R}^m\times\mathbb{R}^m$ defined by
\[
\Tilde{f}=\Tilde{f}_0+
\begin{bmatrix}
\xi \\
\eta
\end{bmatrix},
\qquad
\Tilde{f}_0
=\begin{bmatrix}
f_0 \\
0 
\end{bmatrix}.
\]
By simple algebra conditions \eqref{02eqn:in proof 3.1 (2)}
are rephrased as
\begin{align}
&\langle \Tilde{f},\Tilde{f}\rangle=1
\quad\Leftrightarrow\quad
\langle\xi,\xi\rangle+\langle\eta,\eta\rangle
 +2\langle f_0,\xi\rangle=0, 
\label{02eqn:condition for g 1}\\
&\langle \bm{1},\Tilde{f}\rangle=0
\quad\Leftrightarrow\quad
\langle \bm{1},\xi\rangle+\langle\bm{1},\eta\rangle=0.
\label{02eqn:condition for g 2}
\end{align}
Moreover, we have
\begin{align}
\Phi
&=\langle \Tilde{f}, \Tilde{D}\Tilde{f}\rangle 
=\left\langle 
 \begin{bmatrix}
 f_0+\xi \\ \eta
 \end{bmatrix},
 \begin{bmatrix}
 D & S \\
 S^T & J-I
 \end{bmatrix} 
 \begin{bmatrix}
 f_0+\xi \\ \eta
 \end{bmatrix}
\right\rangle
\nonumber \\[2pt]
&=\lambda_0
 +2\langle f_0, D\xi\rangle
 +2\langle f_0,S\eta\rangle
 +2\langle \xi,S\eta\rangle
 +\langle\xi,D\xi\rangle
 +\langle\bm{1},\eta\rangle^2
 -\langle\eta,\eta\rangle,
\label{2eqn:in proof 3.1 (3)}
\end{align}
where we used the simple identity:
$\langle \eta, J\eta\rangle=\langle \bm{1},\eta\rangle^2$.
Using \eqref{02eqn:def of S} we obtain
\begin{align}
\langle f_0,S\eta\rangle
&=\sum_{i=1}^n f_0(i) S\eta(i)
=\sum_{i=1}^n f_0(i)\sum_{j=1}^m (d(i,n)+1)\eta(j) 
\nonumber \\
&=\sum_{i=1}^n f_0(i)(d_G(i,n)+1)\langle\bm{1},\eta\rangle
=(Df_0(n)+\langle\bm{1},f_0\rangle)\langle\bm{1},\eta\rangle.
\label{in proof 3.1 (6)}
\end{align}
Similarly,
\begin{equation}
\langle \xi,S\eta\rangle
=(D\xi(n)+\langle\bm{1},\xi\rangle)\langle\bm{1},\eta\rangle.
\label{in proof 3.1 (7)}
\end{equation}
Inserting \eqref{in proof 3.1 (6)}
and \eqref{in proof 3.1 (7)} into \eqref{2eqn:in proof 3.1 (3)},
and then applying \eqref{02eqn:condition for g 1},
\eqref{02eqn:condition for g 2} and \eqref{02eqn:in proof 3.1 (12)},
we obtain
\begin{align}
\Phi=\Phi(\xi,\eta)&=\lambda_0 
 +\langle \xi,D\xi\rangle + \langle \xi,\xi\rangle
 -\langle\bm{1},\xi\rangle^2 
\nonumber \\
&\quad +2\langle f_0, D\xi\rangle
 +2\langle f_0,\xi\rangle
 -2Df_0(n)\langle\bm{1},\xi\rangle 
 -2D\xi(n)\langle\bm{1},\xi\rangle,
\label{2eqn:in proof 3.1 (9)}
\end{align}
Thus, $\mathrm{QEC}(\Tilde{G})$ coincides with
the conditional maximum of $\Phi(\xi,\eta)$
subject to \eqref{02eqn:condition for g 1}
and \eqref{02eqn:condition for g 2}.
Here note that $\eta$ is implicitly contained in 
\eqref{2eqn:in proof 3.1 (9)} through those conditions.

To be precise, we put
\[
\mathcal{M}
=\left\{(\xi,\eta)\in\mathbb{R}^n\times\mathbb{R}^m\,;\,
\begin{array}{l}
\langle\xi,\xi\rangle+\langle\eta,\eta\rangle
 +2\langle f_0,\xi\rangle=0, \\
\langle \bm{1},\xi\rangle+\langle\bm{1},\eta\rangle=0
\end{array}
\right\}.
\]
Then we have
\[
\mathrm{QEC}(\Tilde{G})
=\max\{\Phi(\xi,\eta)\,;\, (\xi,\eta)\in\mathcal{M}\}.
\]
Since $(0,0)\in\mathcal{M}$ and
$\Phi(0,0)=\lambda_0=\mathrm{QEC}(G)$,
for $\mathrm{QEC}(G)<\mathrm{QEC}(\Tilde{G})$
it is sufficient to show that $\Phi(\xi,\eta)$ does not
attain a conditional maximum at $(\xi,\eta)=(0,0)$.
We will prove this by contradiction.

Suppose that $\Phi(\xi,\eta)$ attains 
a conditional maximum at $(\xi,\eta)=(0,0)$.
Then the directional derivative of 
$\Phi(\xi,\eta)$ at $(\xi,\eta)=(0,0)$ vanishes 
along any curve in $\mathcal{M}$ passing through $(0,0)$.
For $1\le k\le n-1$ we put
\[
\mathcal{N}_k
=\left\{(\xi=[\xi(i)],\eta=[\eta(j)])\in
\mathbb{R}^n\times\mathbb{R}^m\,;\,
\begin{array}{l}
\text{$\xi(i)=0$ except $i=k$ and $i=n$,} \\
\text{$\eta(j)=0$ except $j=1$} 
\end{array}
\right\}
\]
and 
\[
\mathcal{M}_k
=\mathcal{M}\cap \mathcal{N}_k\,.
\]
From \eqref{02eqn:condition for g 1} and
\eqref{02eqn:condition for g 2} we see that
$(\xi,\eta)\in \mathcal{N}_k$ belongs to $\mathcal{M}$ 
if and only if
\begin{gather}
2f_0(k)\xi(k)+2f_0(n)\xi(n)+\xi(k)^2+\xi(n)^2+\eta(1)^2=0, 
\label{03eqn:in proof 31}\\
\xi(k)+\xi(n)+\eta(1)=0.
\label{03eqn:in proof 32}
\end{gather}
Inserting \eqref{03eqn:in proof 32} into \eqref{03eqn:in proof 31},
we obtain 
\begin{equation}\label{03eqn:in proof 3.1 (11) ellipse}
\xi(k)^2+\xi(n)^2+\xi(k)\xi(n)+f_0(k)\xi(k)+f_0(n)\xi(n)=0,
\end{equation}
which determines an ellipse of positive radius
since $f_0(n)\neq0$ by assumption.
Namely, $\mathcal{M}_k$ is an ellipse 
in $\mathbb{R}^n\times\mathbb{R}^m$
passing through $(0,0)$.

Now consider the directional derivative of $\Phi(\xi,\eta)$ 
at $(\xi,\eta)=0$ along the ellipse $\mathcal{M}_k$.
From \eqref{03eqn:in proof 3.1 (11) ellipse}
we obtain easily that
\[
\frac{d\xi(n)}{d\xi(k)}\bigg|_{(\xi(k),\xi(n))=(0,0)}
=-\frac{f_0(k)+2\xi(k)+\xi(n)}{f_0(n)+\xi(k)+2\xi(n)}
 \bigg|_{(\xi(k),\xi(n))=(0,0)}
=-\frac{f_0(k)}{f_0(n)}.
\]
On the other hand, 
inserting \eqref{03eqn:in proof 31} and \eqref{03eqn:in proof 32}
into \eqref{2eqn:in proof 3.1 (9)},
we see that $\Phi=\Phi(\xi,\eta)$ on $\mathcal{M}_k$ becomes
\begin{align*}
\Phi=\Phi(\xi(k),\xi(n))
&=\lambda_0-2d(n,k)\xi(k)^2-2\xi(k)\xi(n) \\
&\qquad+2(f_0(k)+Df_0(k)-Df_0(n))\xi(k)
+2f_0(n)\xi(n).
\end{align*}
Then again by simple calculus, we come to
\begin{equation}\label{02eqn:in proof 3.1 (22)}
\frac{d\Phi}{d\xi(k)}\bigg|_{(\xi(k),\xi(n))=(0,0)}
=2Df_0(k)-2Df_0(n).
\end{equation}
Since $d\Phi/d\xi(k)$ at $\xi=0$ vanishes
for all $1\le k\le n-1$ by assumption,
it follows from \eqref{02eqn:in proof 3.1 (22)} that
$Df_0(k)=Df_0(n)$ for all $1\le k\le n-1$.
Hence $Df_0=Df_0(n)\bm{1}$ and we come to
\[
\lambda_0
=\langle f_0,Df_0\rangle
=Df_0(n)\langle f_0,\bm{1}\rangle
=0,
\]
which is in contradiction to $\lambda_0=\mathrm{QEC}(G)<0$.
\end{proof}

\begin{proposition}\label{02prop:star product estimate}
Let $G_1$ and $G_2$ be connected graphs with
$\mathrm{QEC}(G_1)<0$
and $\mathrm{QEC}(G_2)<0$.
Then
\begin{equation}\label{02eqn:general estimate for star product}
\mathrm{QEC}(G_1\star G_2)
\le\left(\frac{1}{\mathrm{QEC}(G_1)}
 +\frac{1}{\mathrm{QEC}(G_2)}\right)^{-1}<0.
\end{equation}
\end{proposition}

For the proof see \cite[Section 4]{MO-2018},
where a more precise estimate is obtained.

\begin{proof}[Proof of the right-half of 
\eqref{02eqn:strict inequality}] 
Note that $\mathrm{QEC}(K_{m+1})=-1$ for all $m\ge1$.
It then follows immediately from
Proposition \ref{02prop:star product estimate} that
\[
\mathrm{QEC}(G\star K_m)
\le\left(\frac{1}{\mathrm{QEC}(G)}
 +\frac{1}{-1}\right)^{-1}
=\frac{\mathrm{QEC}(G)}{1-\mathrm{QEC}(G)}<0.
\]
Here condition \eqref{02eqn:crucial condition} is not necessary.
\end{proof}

\begin{remark}\normalfont\rmfamily
For the strict inequality of the left-half of 
\eqref{02eqn:strict inequality} condition
\eqref{02eqn:crucial condition} is necessary.
We give a simple example.
Consider the graph $G$ on five verices
and $\Tilde{G}=G\star K_2$ on six vertices
as is illustrated in Figure \ref{fig:in remark}.
By direct computation we easily obtain
\[
\mathrm{QEC}(G)=\mathrm{QEC}(\Tilde{G})=-\frac{2}{2+\sqrt2}\,.
\]
In fact, $\mathrm{QEC}(G)$ is attained by 
\[
f_0=c
\begin{bmatrix}
\pm1 \\
\mp1 \\
\pm(\sqrt2+1) \\
\mp(\sqrt2+1) \\
0
\end{bmatrix},
\qquad
c=\sqrt{\frac{2-\sqrt2}{8}}.
\]
Indeed, $f_0(5)=0$ and
condition \eqref{02eqn:crucial condition} is fulfilled.
More examples will appear in Subsection
\ref{Bearded Complete Graphs}.
While, it is not clear whether 
$\mathrm{QEC}(G)=\mathrm{QEC}(\Tilde{G})$ follows from
$f_0(n)=0$.
\end{remark}

\begin{figure}[htb]
\begin{center}
\vspace*{10pt}
\includegraphics[height=80pt]{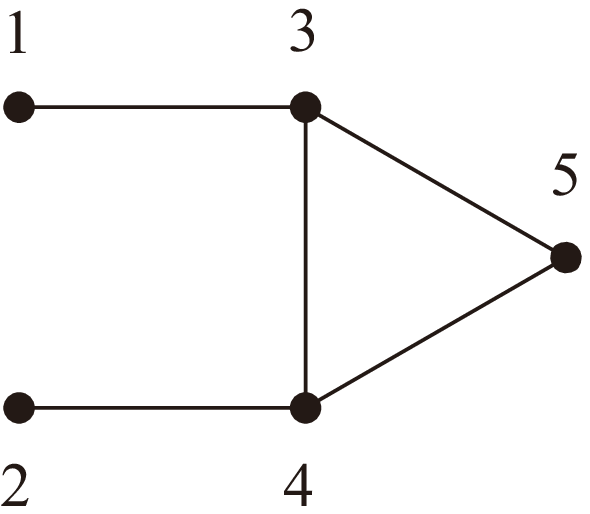}
\qquad\qquad
\includegraphics[height=80pt]{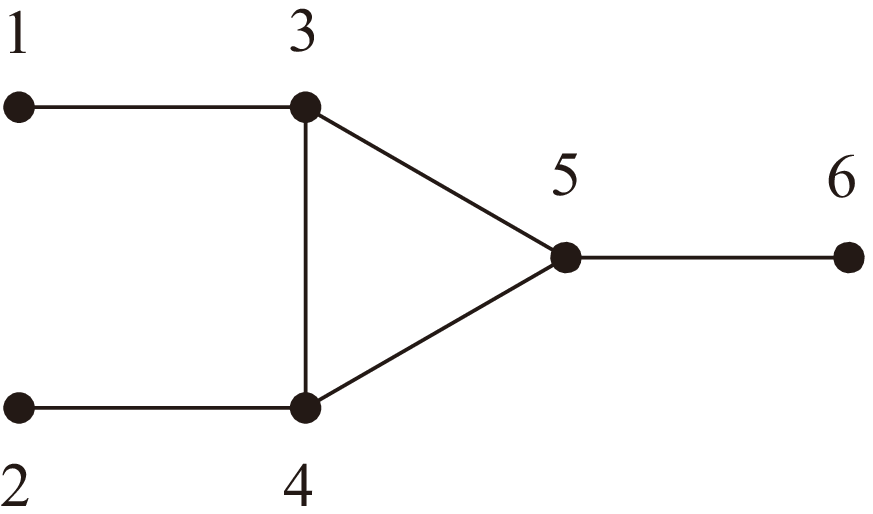}
\end{center}
\caption{An example of $f_0(5)=0$}
\label{fig:in remark}
\end{figure}
%

\subsection{QE Constants of Paths}
\label{03subsec:QE Constants of Paths}

For $n\ge1$ let $P_n$ be the path on $\{1,2,\dots,n\}$.
Since $P_n$ is isometrically embedded in $P_{n+1}$, we have
\[
\mathrm{QEC}(P_n)\le \mathrm{QEC}(P_{n+1}),
\qquad n\ge2.
\]
In this section we prove that the above inequality is strict.

\begin{theorem}\label{03thm:QEC(Path)}
For $n\ge2$ we have
$\mathrm{QEC}(P_n)< \mathrm{QEC}(P_{n+1})$.
\end{theorem}

\begin{proof}
The distance matrix of $P_n$ is given by
$D=[d(i,j)]$ with $d(i,j)=|i-j|$, $1\le i,j\le n$.
According to the general method described in
Subsection \ref{02subsec:Calculating QE Constants} let
$\mathcal{S}$ be the set of 
$(f,\lambda,\mu)\in\mathbb{R}^n\times \mathbb{R}\times\mathbb{R}$
such that
\begin{gather}
(D-\lambda)f=\frac{\mu}{2}\,\bm{1},
\label{04eqn:in proof 4.1 (1)}\\
\langle f,f\rangle=1,
\label{04eqn:in proof 4.1 (2)}\\
\langle \bm{1},f\rangle=0.
\label{04eqn:in proof 4.1 (3)}
\end{gather}
Then $\lambda_0=\mathrm{QEC}(P_n)$ is 
the maximum of $\lambda\in\mathbb{R}$ such that
$(f,\lambda,\mu)\in\mathcal{S}$
for some $f\in\mathbb{R}^n$ and $\mu\in\mathbb{R}$.
It is readily known that $\lambda_0<0$.
By virtue of Theorem \ref{02thm:strict inequality}
it is sufficient to show that there exists 
$(f_0=[f_0(i)],\lambda_0,\mu_0)\in\mathcal{S}$
such that $f_0(n)\neq0$.

In fact, we will prove a slightly stronger result:
for any $(f=[f(i)],\lambda,\mu)\in\mathcal{S}$ we have $f(n)\neq0$.
First assume that $(f,\lambda,\mu)\in\mathcal{S}$ fulfills
\begin{equation}\label{03eqn:f(j)=0}
f(j)=0 \quad \text{for}\quad k\le j\le n,
\end{equation}
where $2\le k\le n$.
We will derive $f(k-1)=0$.
The $k$-th coordinate of \eqref{04eqn:in proof 4.1 (1)} is
given by
\begin{equation}\label{03eqn:in proof 4.1 (6)}
\sum_{j=1}^n |k-j|f(j)-\lambda f(k)=\frac{\mu}{2}
\end{equation}
and by assumption \eqref{03eqn:f(j)=0} we have
\begin{equation}\label{03eqn:in proof 4.1 (8)}
\sum_{j=1}^{k-1} (k-j)f(j)-\lambda f(k)=\frac{\mu}{2}\,.
\end{equation}
Similarly, looking at the $(k-1)$-th coordinate 
of \eqref{04eqn:in proof 4.1 (1)}, we obtain
\begin{equation}\label{04eqn:in proof 4.1 (7)}
\sum_{j=1}^{k-1} (k-1-j)f(j)-\lambda f(k-1)=\frac{\mu}{2}\,.
\end{equation}
On the other hand, by \eqref{04eqn:in proof 4.1 (3)}
and \eqref{03eqn:f(j)=0} we have
\[
\sum_{j=1}^{k-1} f(j)=0.
\]
Then \eqref{04eqn:in proof 4.1 (7)} becomes
\begin{equation}\label{03eqn:in proof 4.1 (9)}
\sum_{j=1}^{k-1} (k-j)f(j)-\lambda f(k-1)=\frac{\mu}{2}\,.
\end{equation}
Comparing \eqref{03eqn:in proof 4.1 (8)} and 
\eqref{03eqn:in proof 4.1 (9)},
we obtain
\[
\lambda(f(k-1)-f(k))=0.
\]
Since $\lambda\le \lambda_0<0$, we 
obtain $f(k-1)=f(k)=0$ as desired.
Thus, by induction we see that $f(n)=0$ implies that
$f(j)=0$ for all $1\le j\le n$,
which is in contradiction to condition \eqref{04eqn:in proof 4.1 (2)}.
Consequently, $f(n)\neq0$ for any $(f,\lambda,\mu)\in\mathcal{S}$.
\end{proof}

\begin{proposition}\label{03prop:lim odf P_n}
We have
\[
\lim_{n\rightarrow\infty}\mathrm{QEC}(P_n)=-\frac12\,.
\]
\end{proposition}

For the proof see \cite[Section 5]{MO-2018},
where a precise estimate of $\mathrm{QEC}(P_n)$ from below is
obtained.

\section{Classification of Graphs Along $\mathrm{QEC}(P_n)$}
\label{sec:Classification}

\subsection{Formulation of Problem}

Combining Theorem \ref{03thm:QEC(Path)} and 
Proposition \ref{03prop:lim odf P_n}, we come to
\begin{equation}\label{04eqn:basic scale}
\mathrm{QEC}(P_2)<
\mathrm{QEC}(P_3)<\cdots<
\mathrm{QEC}(P_n)<
\mathrm{QEC}(P_{n+1})<
\cdots\rightarrow -\frac12.
\end{equation}
In fact, the first few are given as follows:
\begin{align*}
&\mathrm{QEC}(P_2)=-1,\\
&\mathrm{QEC}(P_3)=-\frac23=-0.6666\cdots,\\
&\mathrm{QEC}(P_4)=-\frac{2}{2+\sqrt2}=-(2-\sqrt2)=-0.5857\cdots,\\
&\mathrm{QEC}(P_5)=-\frac{4}{5+\sqrt5}
=-\frac{5-\sqrt5}{5}=-0.5527\cdots,\\
&\mathrm{QEC}(P_6)=-\frac{2}{2+\sqrt3}=-(4-2\sqrt3)=-0.5358\cdots.
\end{align*}
A closed formula for $\mathrm{QEC}(P_n)$ is not known.

Our main interest along \eqref{04eqn:basic scale}
is to characterize the family of graphs $G$ satisfying
\begin{equation}\label{04eqn:class n}
\mathrm{QEC}(P_n)\le \mathrm{QEC}(G)<\mathrm{QEC}(P_{n+1}),
\qquad n\ge2,
\end{equation}
in terms of geometric or combinatorial properties of graphs.
We are also interested in the graphs $G$ satisfying
\begin{equation}\label{04eqn:QEC<-1/2}
\mathrm{QEC}(G)<-\frac12\,.
\end{equation}

We first recall the following simple fact
mentioned in Corollary \ref{02cor:QEC and diameter} (2).

\begin{proposition}\label{03prop:diam}
Let $n\ge2$.
If $\mathrm{QEC}(G)<\mathrm{QEC}(P_{n+1})$, 
then $\mathrm{diam}(G)\le n-1$.
\end{proposition}

Next we provide simple criteria for
\eqref{04eqn:QEC<-1/2} in terms of forbidden subgraphs.
Let $K_4\backslash\{e\}$ denote the \textit{diamond},
that is, the graph obtained by deleting one edge
from the complete graph $K_4$,
see Figure \ref{fig:subgraphs for -1/2}.
Let $K_{m,n}$ denote the complete bipartite graph with
two parts of $m$ and $n$ vertices.
In particular, $K_{1,n}$ is called a \textit{star}
and $K_{1,3}$ a \textit{claw},
see Figure \ref{fig:subgraphs for -1/2}.
\begin{figure}[htb]
\begin{center}
\vspace*{10pt}
\includegraphics[width=70pt]{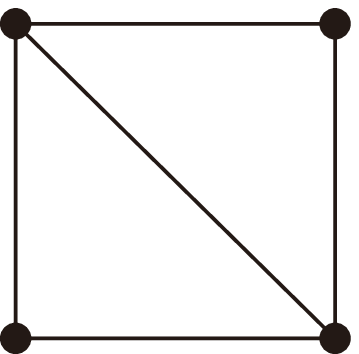}
\hspace*{70pt}
\includegraphics[width=70pt]{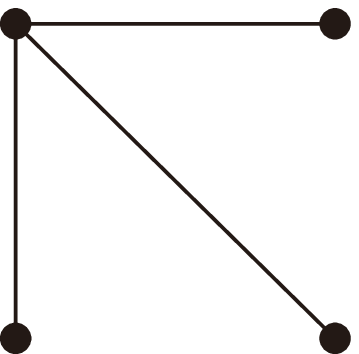}
\caption{$K_4\backslash\{e\}$ (diamond) and $K_{1,3}$ (claw)}
\label{fig:subgraphs for -1/2}
\end{center}
\end{figure}

\begin{proposition}\label{04prop:sufficient condition for QEC>=-1/2}
If a connected graph $G$ contains an induced subgraph
isomorphic to a diamond
$K_4\backslash \{e\}$ or a claw $K_{1,3}$,
then $\mathrm{QEC}(G)\ge -1/2$.
\end{proposition}

\begin{proof}
It is easily verified \cite[Section 5]{Obata-Zakiyyah2018} that 
\[
\mathrm{QEC}(K_4\backslash \{e\})=\mathrm{QEC}(K_{1,3})=-\frac12\,.
\]
Moreover we have 
$\mathrm{diam}(K_4\backslash K_2)=\mathrm{diam}(K_{1,3})=2$.
It then follows from 
Proposition \ref{02prop:QEC estimated from below by diam 2}
that $\mathrm{QEC}(G)\ge -1/2$.
\end{proof}

\begin{proposition}\label{04prop:condition by cycles}
If a connected graph $G$ contains an induced subgraph
isomorphic to the cycle $C_4$, then $\mathrm{QEC}(G)\ge 0$.
If $G$ contains an induced subgraph
isomorphic to $C_5$, then 
\[
\mathrm{QEC}(G)\ge -\frac{2}{3+\sqrt5}=-0.3819\dots.
\]
\end{proposition}

\begin{proof}
We note that
\[
\mathrm{QEC}(C_4)=0,
\qquad
\mathrm{QEC}(C_5)=-\frac{2}{3+\sqrt5},
\]
see also \eqref{02eqn:QEC(C_n)}.
Then the assertion follows in a similar manner
as in the proof of
Proposition \ref{04prop:sufficient condition for QEC>=-1/2}.
\end{proof}

The following result is immediate from
Propositions \ref{04prop:sufficient condition for QEC>=-1/2}
and \ref{04prop:condition by cycles}. 

\begin{corollary}[forbidden subgraphs]
\label{04cor:forbidden subgraphs}
Any graph with $\mathrm{QEC}(G)< -1/2$ does not contain an
induced subgraph isomorphic to
a diamond $K_4\backslash\{e\}$,
a claw $K_{1,3}$,
a cycle $C_4$,
nor $C_5$.
In short,
any graph with $\mathrm{QEC}(G)< -1/2$ is
diamond-free, claw-free, $C_4$-free and $C_5$-free.
\end{corollary}

\begin{remark}\normalfont\rmfamily
As an immediate consequence from
Corollary \ref{04cor:forbidden subgraphs},
the family of graphs with $\mathrm{QEC}(G)< -1/2$ 
forms a subfamily of the claw-free graphs.
On the other hand, 
claw-free graphs has been actively studied
with various classifications, see e.g., \cite{FFR1997}.
It would be interesting to revisit the classification
of claw-free graphs along with $\mathrm{QEC}(P_n)$.
\end{remark}

\subsection{Determining the class 
$\mathrm{QEC}(P_2)\le \mathrm{QEC}(G)<\mathrm{QEC}(P_3)$}

\begin{theorem}\label{05thm: G_2}
For a connected graph $G$ the inequality 
\begin{equation}\label{05eqn:condition for G_2}
\mathrm{QEC}(P_2)\le \mathrm{QEC}(G)<\mathrm{QEC}(P_3)
\end{equation}
holds if and only if $G=K_n$ for some $n\ge2$.
Moreover, $\mathrm{QEC}(P_2)= \mathrm{QEC}(K_n)$ for all $n\ge2$.
Therefore, there is no graph
$G$ such that $\mathrm{QEC}(P_2)<\mathrm{QEC}(G)<\mathrm{QEC}(P_3)$.
\end{theorem}

\begin{proof}
Suppose that a graph $G$ satisfies \eqref{05eqn:condition for G_2}.
Then by Proposition \ref{03prop:diam},
we have $\mathrm{diam}(G)=1$,
which means that $G$ is a complete graph.
On the other hand, it is known that
$\mathrm{QEC}(K_n)=-1=\mathrm{QEC}(P_2)$ for all $n\ge2$.
The assertion is then obvious.
\end{proof}


\subsection{Calculating $\mathrm{QEC}(K_n\star K_m)$}

We consider the star product of two complete graphs $K_n$ and $K_m$,
see Figure \ref{05fig:K_n*K_m}.
To be precise, let $n\ge1$ and $m\ge2$,
and consider the graphs 
$\Tilde{G}=(\Tilde{V},\Tilde{E})$,
where
\[
\Tilde{V}=\{1,2,\dots,n\}\cup\{n,n+1,\dots,n+m-1\}
\]
and
\[
\Tilde{E}=\{\{i,j\}\,;\, 1\le i<j\le n\}\cup
\{\{i,j\}\,;\, n\le i<j\le n+m-1\}.
\]
Obviously, we have $\Tilde{G}=K_n\star K_m$,
where the induced subgraphs spanned by
$\{1,2,\dots,n\}$ and by $\{n,n+1,\dots,n+m-1\}$ are
the complete graphs $K_n$ and $K_m$, respectively.
\begin{figure}[htb]
\begin{center}
\vspace*{10pt}
\includegraphics[height=90pt]{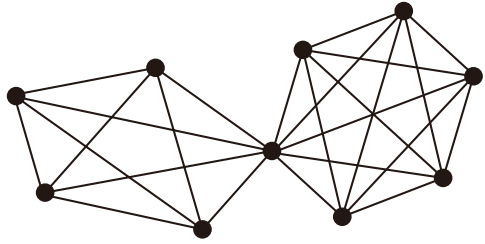}
\caption{$K_n\star K_m$ ($n=5, m=6$)}
\label{05fig:K_n*K_m}
\end{center}
\end{figure}

Let $\Tilde{D}$ be the distance matrix of 
$\Tilde{G}=K_n\star K_m$.
It is convenient to write $\Tilde{D}$ in the block matrices:
\begin{equation}\label{05eqn:D of K_n star K_m}
\Tilde{D}
=\left[
\renewcommand{\arraystretch}{2.4}
\begin{array}{@{\,}ccc |ccc@{\,}}
 & J-I & & & S & \\ \hline
 & S^T & & & J-I & \\
\end{array}
\right],
\qquad
S=
\begin{bmatrix}
2 & \cdots & 2 \\
\vdots &  & \vdots \\
2 & \cdots & 2 \\
1 & \cdots & 1 
\end{bmatrix},
\end{equation}
where $S$ is an $n\times (m-1)$ matrix.
The QE constant $\mathrm{QEC}(\Tilde{G})$ is the conditional 
maximum of 
\begin{equation}\label{04eqn:Phi for K_m*K_n}
\Phi=\langle \Tilde{f}, \Tilde{D}\Tilde{f}\rangle,
\qquad
\Tilde{f}\in C(\Tilde{V}),
\end{equation}
subject to 
\[
\langle \Tilde{f}, \Tilde{f}\rangle=1,
\qquad
\langle \bm{1}, \Tilde{f}\rangle=0.
\]
According to the block diagonal 
expression \eqref{05eqn:D of K_n star K_m}, we write
$\Tilde{f}=[f\,\,\,g]^T$, where $f\in \mathbb{R}^n$,
$g\in \mathbb{R}^{m-1}$.
Then \eqref{04eqn:Phi for K_m*K_n} becomes
\begin{align*}
\Phi
&=\Phi(f,g)
=\left\langle 
 \begin{bmatrix} f \\ g \end{bmatrix},
 \begin{bmatrix} J-I & S \\ S^T & J-I \end{bmatrix}
 \begin{bmatrix} f \\ g \end{bmatrix},
 \right\rangle
\\
&=\langle 1,f\rangle^2
 +\langle 1,g\rangle^2
 -\langle f,f\rangle
 -\langle g,g\rangle
 +4\langle \bm{1},f\rangle \langle \bm{1}, g\rangle
 -2f_n \langle \bm{1}, g\rangle,
\end{align*}
where we used 
\[
Sg=\langle\bm{1},g\rangle 
[2 \,\, 2 \cdots \,\, 2 \,\, 1 ]^T.
\]
Define
\[
F(f,g,\lambda,\mu)=
\Phi(f,g)
-\lambda(\langle f,f\rangle+\langle g,g\rangle-1)
-\mu(\langle \bm{1},f\rangle+\langle \bm{1},g\rangle)
\]
and let $\mathcal{S}$ be the set of its
stationary points $(f,g,\lambda,\mu)\in
\mathbb{R}^n\times\mathbb{R}^{m-1}\times \mathbb{R}\times\mathbb{R}$,
that is the solutions to
\begin{equation}\label{04eqn:stationary points equation}
\frac{\partial F}{\partial f_i}
=\frac{\partial F}{\partial g_j}
=\frac{\partial F}{\partial \lambda}
=\frac{\partial F}{\partial \mu}
=0,
\quad 1\le i\le n,
\,\,\,
1\le j\le m-1.
\end{equation}
Keeping in mind that
$-1<\mathrm{QEC}(\Tilde{G})<0$ unless $m=1$ or $n=1$,
we find after simple calculus that
the maximum of $\lambda$ appearing in the solution is
\[
\lambda=\frac{-mn+\sqrt{mn(m-1)(n-1)}}{m+n-1}
=-\frac{1}
 {1+\sqrt{\Big(1-\dfrac{1}{m}\Big)\Big(1-\dfrac{1}{n}\Big)}}\,,
\]
which coincides with $\mathrm{QEC}(\Tilde{G})$ by
the general theory mentioned in Subsection
\ref{02subsec:Calculating QE Constants}.
We have thus obtained the following result.

\begin{proposition}\label{04prop:QEC(K_n*K_m)}
For $m\ge1$ and $n\ge1$ with $m+n\ge3$ with we have
\[
\mathrm{QEC}(K_n\star K_m)
=-\frac{1}
 {1+\sqrt{\Big(1-\dfrac{1}{m}\Big)\Big(1-\dfrac{1}{n}\Big)}}\,.
\]
\end{proposition}

\begin{corollary}
We have
\begin{align}
\mathrm{QEC}(P_3)
&=\mathrm{QEC}(K_2\star K_2)
<\mathrm{QEC}(K_3\star K_2)<\dotsb 
\nonumber \\
&\dotsb<\mathrm{QEC}(K_n\star K_2)<\dotsb
\rightarrow \mathrm{QEC}(P_4)=-\frac{2}{2+\sqrt{2}}.
\end{align}
\end{corollary}

\begin{proof}
By Proposition \ref{04prop:QEC(K_n*K_m)} we have
\[
\mathrm{QEC}(K_n\star K_2)
=-\frac{2}{2+\sqrt{2\Big(1-\dfrac{1}{n}\Big)}}\,,
\quad n\ge1,
\]
from which the assertion follows immediately.
\end{proof}

\begin{corollary}
Let $m\ge1$ and $n\ge1$ with $m+n\ge3$.
Then $\mathrm{QEC}(K_n\star K_m)<\mathrm{QEC}(P_4)$ 
if and only if one of the following conditions is 
satisfied:
\begin{enumerate}
\setlength{\itemsep}{0pt}
\item[\upshape (i)] $m=2$ and $n\ge1$;
\item[\upshape (ii)] $m\ge1$ and $n=2$;
\item[\upshape (iii)] $m=n=3$.
\end{enumerate}
\end{corollary}

\begin{proof}
The inequality $\mathrm{QEC}(K_n\star K_m)<\mathrm{QEC}(P_4)$ 
is equivalent to 
\[
-\frac{1}
 {1+\sqrt{\Big(1-\dfrac{1}{m}\Big)\Big(1-\dfrac{1}{n}\Big)}}
<-\frac{2}{2+\sqrt{2}}\,,
\]
of which integer solutions are obtained easily 
by simple algebra.
\end{proof}

\begin{corollary}\label{04cor:halfly determining G_3} 
Let $m\ge1$ and $n\ge1$ with $m+n\ge3$.
Then 
\begin{equation}\label{4eqn:between P3 and P4}
\mathrm{QEC}(P_3)\le \mathrm{QEC}(K_n\star K_m)<\mathrm{QEC}(P_4)
\end{equation}
holds if and only if one of the following conditions is 
satisfied:
\begin{enumerate}
\setlength{\itemsep}{0pt}
\item[\upshape (i)] $m=2$ and $n\ge2$;
\item[\upshape (ii)] $m\ge2$ and $n=2$;
\item[\upshape (iii)] $m=n=3$.
\end{enumerate}
The equality in \eqref{4eqn:between P3 and P4} 
occurs only when $m=n=2$.
\end{corollary}


\subsection{Determining the class 
$\mathrm{QEC}(P_3)\le \mathrm{QEC}(G)<\mathrm{QEC}(P_4)$}

This subsection is devoted to the proof of the following result.

\begin{theorem}\label{05thm: G_3}
A finite connected graph $G$ fulfills the inequality 
\begin{equation}\label{04eqn:condition G_3}
\mathrm{QEC}(P_3)\le \mathrm{QEC}(G)<\mathrm{QEC}(P_4)
\end{equation}
if and only if $G$ is a star product $K_n \star K_2$ with 
$n\ge2$ or $K_3\star K_3$.
Moreover,
\begin{align*}
\mathrm{QEC}(K_n\star K_2)
&=-\frac{2}{2+\sqrt{2\Big(1-\dfrac1n\Big)}}\,, \\
\mathrm{QEC}(K_3\star K_3)
&=-\frac{3}{5}\,.
\end{align*}
In particular, $\mathrm{QEC}(G)=\mathrm{QEC}(P_3)$
if and only if $G=P_3=K_2\star K_2$.
\end{theorem}

\begin{lemma}\label{04lem:01}
If a connected graph $G=(V,E)$ satisfies
\eqref{04eqn:condition G_3}, we have
$|V|\ge3$ and $\mathrm{diam}(G)=2$.
\end{lemma}

\begin{proof}
It follows from Corollary \ref{02cor:QEC and diameter} that
$\mathrm{diam}(G)\le2$.
If $\mathrm{diam}(G)=1$, then $G$ is a complete graph 
and $\mathrm{QEC}(G)=-1$, which
does not satisfy \eqref{04eqn:condition G_3}.
Hence, necessarily $\mathrm{diam}(G)=2$ and $|V|\ge3$.
\end{proof}

In general, a \textit{clique} of $G$ is an induced subgraph of
$G$ which is isomorphic to a complete graph.
A clique $K=(W,F)$ is called \textit{maximal} if
there is no clique containing $K$ properly.
A maximal clique $K=(W,F)$ is called \textit{largest} 
or \textit{maximum} if
there is no clique on $|W|+1$ vertices.
Clearly, any graph contains a largest clique.

\begin{lemma}\label{04lem:02}
Let $G=(V,E)$ be a connected graph 
satisfying \eqref{04eqn:condition G_3}.
If $K=(W,F)$ is a maximal clique of $G$,
we have $W\neq V$ and $|W|\ge2$.
\end{lemma}

\begin{proof}
Since $G$ is not a complete graph by Lemma \ref{04lem:01},
we have $W\neq V$.
That $|W|\ge2$ follows from $|V|\ge3$.
\end{proof}

\begin{lemma}\label{04lem:out of max clique}
Let $G=(V,E)$ be a connected graph with $|V|\ge2$ and
$\mathrm{QEC}(G)<-1/2$,
and $K=(W,F)$ a maximal clique.
Then for any pair $a\in V\backslash W$ and
$a^\prime \in W$ with $a\sim a^\prime$ we have
$\{x\in W\,;\, x\sim a\}=\{a^\prime\}$.
\end{lemma}

\begin{proof}
(Note that the assertion is trivial if $W=V$.)
Given a pair $a\in V\backslash W$ and
$a^\prime \in W$ with $a\sim a^\prime$,
we set $s=|\{x\in W\,;\, x\sim a\}|$.
Obviously, $1\le s<|W|$.
We will show by contradiction that $s=1$.
Suppose that $s\ge2$.
Then there exist three distinct vertices $x_1, x_2, y\in W$ such that
$a\sim x_1$, $a\sim x_2$ and $a\not \sim y$.
Note that the induced subgraph spanned by $\{a,x_1,x_2,y\}$
is isomorphic to a diamond $K_4\backslash \{e\}$.
It then follows immediately from
Proposition \ref{04prop:sufficient condition for QEC>=-1/2}
that $\mathrm{QEC}(G)\ge -1/2$,
which is in contradiction to the assumption
$\mathrm{QEC}(G)<-1/2$.
\end{proof}

\begin{lemma}\label{05lem:unique a^prime}
Let $G=(V,E)$ be a connected graph 
satisfying \eqref{04eqn:condition G_3}
and $K=(W,F)$ a maximal clique of $G$.
For $a,b\in V\backslash W$ and $a^\prime,b^\prime\in W$,
if $a\sim a^\prime$ and $b\sim b^\prime$, then $a^\prime=b^\prime$.
\end{lemma}

\begin{proof}
If $a=b$ the assertion follows immediately from
Lemma \ref{04lem:out of max clique}.
We consider the case of $a\neq b$.
To prove the assertion by contradiction,
we assume that $a^\prime\neq b^\prime$.
Since $d(a,b)\le \mathrm{diam}(G)=2$,
we have two cases: $d(a,b)=1$ or $d(a,b)=2$.

Suppose first that $d(a,b)=1$, that is, $a\sim b$.
Then the induced subgraph spanned by $\{a,a^\prime,b^\prime, b\}$
is isomorphic to $C_4$,
which is a forbidded subgraph by
Corollary \ref{04cor:forbidden subgraphs}.
Hence $d(a,b)=1$ does not happen.

Suppose next that $d(a,b)=2$.
Then there exists $c\in V$ such that $a\sim c \sim b$.
Since $a\not\sim b^\prime$ and $b\not\sim a^\prime$ by
Lemma \ref{04lem:out of max clique},
we have $c\neq a^\prime,b^\prime$ and $c\not\in W$.
There are four cases:

(i) $c\not\sim a^\prime$ and $c\not\sim b^\prime$.
The induced subgraph spanned by $\{a,a^\prime,b^\prime, b,c\}$
is isomorphic to $C_5$,
which is a forbidded subgraph by
Corollary \ref{04cor:forbidden subgraphs}.

(ii) $c\not\sim a^\prime$ and $c\sim b^\prime$.
The induced subgraph spanned by $\{a,a^\prime,b^\prime,c\}$
is isomorphic to $C_4$,
which is a forbidded subgraph by
Corollary \ref{04cor:forbidden subgraphs}.

(iii) $c\sim a^\prime$ and $c\not\sim b^\prime$.
This case is similar to (ii).

(iv) $c\sim a^\prime$ and $c\sim b^\prime$.
This does not happen by virtue of
Lemma \ref{04lem:out of max clique}.

\par\noindent
In any case we come to contradiction and the proof is completed.
\end{proof}

\begin{proof}[Proof of Theorem \ref{05thm: G_3}]
Let $G=(V,E)$ be a connected graph 
satisfying \eqref{04eqn:condition G_3}
$K=(W,F)$ be a largest clique with $m=|W|$.
Note that $V\neq W$ and $m\ge2$ by Lemma \ref{04lem:02}.
Now divide $V\backslash W$ into two subsets:
\[
V\backslash W=U_1\cup U_2\,,
\]
where $U_1$ is the set of vertices $a\in 
V\backslash W$ which are directly connected 
to vertices in $W$,
and $U_2$ the rest,
see Figure \ref{05fig:WandU_1andU_2}.
Obviously, $U_1\neq\emptyset$.
Moreover, by Lemma \ref{05lem:unique a^prime} there 
exists a unique $a^\prime \in W$ such that
$a\sim a^\prime$ for all $a\in U_1$.

We first prove that $U_2=\emptyset$.
Suppose otherwise.
Take $x\in W$ with $x\neq a^\prime$ and $y\in U_2$.
Then we have $d(x,y)\ge3$, which is in
contradiction to $\mathrm{diam}(G)=2$.

We next prove that any pair of vertices $a,b\in U_1$,
$a\neq b$, are connected by an edge.
Suppose otherwise.
Take $x\in W$ with $x\neq a^\prime$ and 
consider the induced subgraph spanned by
$\{x,a^\prime,a,b\}$ is isomorphic to $K_{1,3}$,
which is a forbidded subgraph by
Corollary \ref{04cor:forbidden subgraphs}.

Consequently, The induced subgraph spanned by $U_1$
is a complete graph on $|U_1|\ge1$ vertices.
Hence $G$ is necessarily a star product of
two complete graphs: $G=K_m\star K_{|U_1|+1}$.
Then the assertion follows from
Corollary \ref{04cor:halfly determining G_3}.
\end{proof}

\begin{figure}[htb]
\begin{center}
\vspace*{10pt}
\includegraphics[height=100pt]{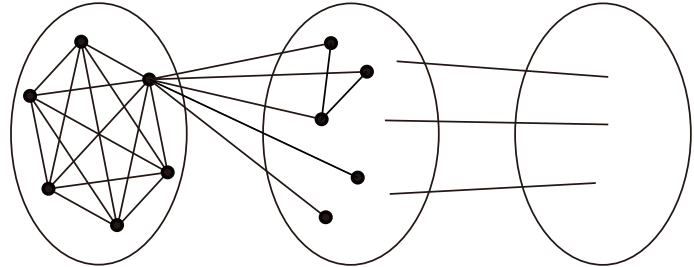}
\caption{$V=W\cup U_1\cup U_2$}
\label{05fig:WandU_1andU_2}
\end{center}
\end{figure}

\subsection{Bearded Complete Graphs $BK_{n,m}$}
\label{Bearded Complete Graphs}

We are also interested in characterization of a graph $G$
satisfying 
\[
\mathrm{QEC}(G)=\mathrm{QEC}(P_4)=-\frac{2}{2+\sqrt2}\,.
\]
Below we give a partial answer.

Let $1\le m\le n$.
Consider a graph on 
\[
V=\{1,2,\dots,n\}\cup\{n+1,\dots,n+m\}
\]
with edge set
\[
E=\{\{i,j\}\,;\, 1\le i<j\le n\}
 \cup\{\{i,n+i\}\,;\, 1\le i\le m\}.
\]
The induced subgraph spanned by
$\{1,2,\dots,n\}$ is the complete graph $K_n$.
We write $G=BK_{n,m}$ and call it  a \textit{bearded complete graph}.

\begin{figure}[htb]
\begin{center}
\vspace*{10pt}
\includegraphics[height=120pt]{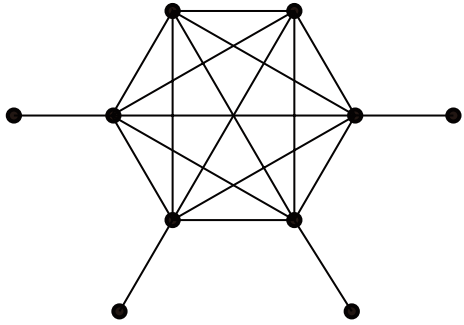}
\caption{$BK_{n,m}$ ($n=6, m=4$)}
\label{05fig:BK_nm}
\end{center}
\end{figure}

The distance matrix $D$ of $G=BK_{n,m}$ is written in the block matrices:
\begin{equation}\label{05eqn:D of BK_nm}
D=\left[
\renewcommand{\arraystretch}{2.4}
\begin{array}{@{\,}ccc |ccc|ccc@{\,}}
 & J-I & & & J   & & & 2J-I &\\ \hline
 & J   & & & J-I & & & 2J &\\ \hline
 & 2J-I & & & 2J & & & 3J-3I & \\
\end{array}
\right],
\end{equation}
where the diagonal matrices are of
$m\times m$, $(n-m)\times(n-m)$ and $m\times m$, in order.

For $m=n=1$ by definition $BK_{1,1}=K_2$.
Hence
\[
\mathrm{QEC}(BK_{1,1})=-1.
\]
For $m=1$ and $n\ge2$ we have $BK_{n,1}=K_n\star K_2=K_n\wedge K_{1,1}$.
It is already known that
\[
\mathrm{QEC}(BK_{n,1})=-\frac{2}{2+\sqrt{2\Big(1-\dfrac{1}{n}\Big)}}\,.
\]
The above formula is valid for $n=1$.

\begin{theorem}
Let $2\le m\le n$.
Then
\[
\mathrm{QEC}(BK_{n,m})=-\frac{2}{2+\sqrt2}=-(2-\sqrt2)
=\mathrm{QEC}(P_4).
\]
\end{theorem}

\begin{proof}
According to the expression \eqref{05eqn:D of BK_nm} in block
diagonal form, the quadratic function
$\Phi=\langle\Tilde{f},D\Tilde{f}\rangle$ becomes
\begin{align*}
\Phi
&=\left\langle 
\begin{bmatrix} f \\ g \\ h \end{bmatrix},
\begin{bmatrix} J-I & J & 2J-I \\
                J & J-I & 2J \\ 
                2J-I & 2J & 3J-3I \end{bmatrix},
\begin{bmatrix} f \\ g \\ h \end{bmatrix}
\right\rangle \\
&=\langle\bm{1},f\rangle^2
  +\langle\bm{1},g\rangle^2
  +3\langle\bm{1},h\rangle^2
   -\langle f,f \rangle
   -\langle g,g \rangle
   -3\langle h,h \rangle \\
&\qquad   +2\langle \bm{1}, f \rangle \langle \bm{1}, g \rangle
   +4\langle \bm{1}, f \rangle \langle \bm{1}, h \rangle
   +4\langle \bm{1}, g \rangle \langle \bm{1}, h \rangle
   -2\langle f,h \rangle,
\end{align*}
where
\[
\Tilde{f}=
\begin{bmatrix}f \\ g \\ h \end{bmatrix},
\qquad
f\in\mathbb{R}^m,
\quad
g\in\mathbb{R}^{n-m},
\quad
h\in\mathbb{R}^m.
\]
We then consider the stationary points of
\begin{align*}
F(f,g,h,\lambda,\mu)
&=\Phi-\lambda(\langle f,f \rangle+\langle g,g \rangle+\langle h,h \rangle-1) \\
&\qquad
-\mu(\langle\bm{1},f\rangle+\langle\bm{1},g\rangle+\langle\bm{1},h\rangle).
\end{align*}
After simple calculus we see that
the largest $\lambda$ appearing in the stationary points 
of $F(f,g,h,\lambda,\mu)$ is given by 
\[
\lambda=-(2-\sqrt2),
\]
with
\[
\langle \bm{1}, f\rangle=0,
\quad
\langle f, f\rangle=\frac{2+\sqrt2}{4}\,,
\quad
g=0,
\quad
h=-(\lambda+1)f,
\quad
\mu=0.
\]
Indeed,
by virtue of the condition $m\ge2$,
we may choose $f\in\mathbb{R}^m$ satisfying
the first two conditions.
\end{proof}



\begin{thebibliography}{99}

\bibitem{Aalipour-etal2016}
G. Aalipour, A. Aida, Z. Berikkyzy et al.:
\textit{On the distance spectra of graphs},
Linear Algebra Appl. {\bfseries 497} (2016), 66--87.

\bibitem{Aouchiche-Hansen2014}
M. Aouchiche and P. Hansen:
\textit{Distance spectra of graphs: a survey},
Linear Algebra Appl. {\bfseries 458} (2014), 301--386.

\bibitem{Balaji-Bapat2007}
R. Balaji and R. B. Bapat:
\textit{On Euclidean distance matrices},
Linear Algebra Appl. {\bfseries 424} (2007), 108--117.

\bibitem{Bapat2010}
R. B. Bapat: ``Graphs and Matrices,"
Springer, Hindustan Book Agency, New Delhi, 2010.


\bibitem{Bozejko88}
M. Bo\.zejko, T. Januszkiewicz and R. J. Spatzier:
\textit{Infinite Coxeter groups do not have Kazhdan's property},
J. Operator Theory 19 (1988), 63--67.

\bibitem{Bozejko89}
M. Bo\.zejko:
\textit{Positive-definite kernels, length functions on groups
and noncommutative von Neumann inequality},
Studia Math. \textbf{95} (1989), 107--118.

\bibitem{Deza-Laurent1997}
M. M. Deza and M. Laurent:
``Geometry of Cuts and Metrics,"
Springer-Verlag, Berlin, 1997. 

\bibitem{FFR1997}
R. Faudree, E. Flandrin and Z. Ryj\'a\v{c}ek:
Claw-free graphs -- A survey,
Discrete Math. {\bfseries 164} (1997), 87--147. 

\bibitem{Haagerup79}
U. Haagerup:
\textit{An example of a nonnuclear C$^*$-algebra
which has the metric approximation property},
Invent. Math. \textbf{50} (1979), 279--293.

\bibitem{Indulal-Gutman2008}
G. Indulal and I. Gutman:
\textit{On the distance spectra of some graphs},
Mathematical Communications {\bfseries 13} (2008), 123--131.

\bibitem{Jaklic-Modic2010}
G. Jakli\v{c} and J. Modic:
\textit{On properties of cell matrices}.
Appl. Math. Comput. {\bfseries 216} (2010), 2016--2023.

\bibitem{Jaklic-Modic2013}
G. Jakli\v{c} and J. Modic:
\textit{On Euclidean distance matrices of graphs},
Electron. J. Linear Algebra {\bfseries 26} (2013), 574--589.

\bibitem{Jaklic-Modic2014}
G. Jakli\v{c} and J. Modic:
\textit{Euclidean graph distance matrices of generalizations
of the star graph},
Appl. Math. Comput. {\bfseries 230} (2014), 650--663.


\bibitem{Koolen-Shpectorov1994}
J. H. Koolen and S. V. Shpectorov:
\textit{Distance-regular graphs the distance matrix 
of which has only one positive eigenvalue},
European J. Combin. {\bfseries 15} (1994), 269--275.

\bibitem{Liberti-Lavor-Maculan-Mucherino2014}
L. Liberti, G. Lavor, N. Maculan and A. Mucherino:
\textit{Euclidean distance geometry and applications},
SIAM Rev. {\bfseries 56} (2014), 3--69.

\bibitem{Liu-Xue-Guo2015}
R. Liu, J. Xue and L. Guo:
\textit{On the second largest distance eigenvalue of a graph},
arXiv:1504.04225v1, 2015.

\bibitem{MO-2018}
W. M\l otkowski and N. Obata:
On quadratic embedding constants of star product graphs,
to appear in Hokkaido Math. J.
arXiv:1802.01214

\bibitem{Obata2007}
N. Obata: \textit{Positive Q-matrices of graphs},
Studia Math. {\bfseries 179} (2007), 81--97.

\bibitem{Obata2011}
N. Obata:
\textit{Markov product of positive definite kernels 
and applications to Q-matrices of graph products},
Colloq. Math. {\bfseries 122} (2011), 177--184.

\bibitem{Obata2017}
N. Obata:
\textit{Quadratic embedding constants of wheel graphs},
Interdiscip. Inform. Sci. {\bfseries 23} (2017), 171--174.

\bibitem{Obata-Zakiyyah2018}
N. Obata and A. Y. Zakiyyah:
\textit{Distance matrices and quadratic embedding of graphs},
Electronic J. Graph Theory Appl. {\bfseries 6} (2018), 37--60.

\bibitem{Schoenberg1935}
I. J. Schoenberg:
\textit{Remarks to Maurice Fr\'echet's article
``Sur la d\'efinition axiomatique d'une
classe d'espace distanci\'s vectoriellement applicable sur l'espace de Hilbert'',}
Ann. of Math. {\bfseries 36} (1935), 724--732.

\bibitem{Schoenberg1938}
I. J. Schoenberg:
\textit{Metric spaces and positive definite functions},
Trans. Amer. Math. Soc. {\bfseries 44} (1938), 522--536.

\bibitem{Young-Householder1938}
G. Young and A. S. Householder:
\textit{Discussion of a set of points in terms of their mutual distances},
Psychometrika {\bfseries 3} (1938), 1--22.

\end{thebibliography}
\end{document}